\documentclass[12pt]{amsart}
\usepackage[utf8]{inputenc}
\usepackage[margin=1in]{geometry}
\usepackage{amsmath,amssymb,amsthm}
\usepackage{tikz-cd}
\usepackage{graphicx}
\usepackage{hyperref}
\usepackage{color}
\DeclareMathOperator{\Ricc}{Ric}
\DeclareMathOperator{\grad}{grad}

\DeclareMathOperator{\Hess}{Hess}
\DeclareMathOperator{\trl}{\Delta_B}
\DeclareMathOperator{\tr}{tr}
\DeclareMathOperator{\ind}{ind}
\DeclareMathOperator{\Ker}{Ker}
\DeclareMathOperator{\Id}{Id}
\DeclareMathOperator{\Herm}{Herm}
\DeclareMathOperator{\Fix}{Fix}
\newcommand{\Gr}{\mathrm{Gr}_{\mathbb{C}}(k,N)}
\newcommand{\Grtf}{\mathrm{Gr}_{\mathbb{C}}(2,4)}

\newcommand{\J}{\mathcal{J}}

\newtheorem*{definition}{Definition}
\newtheorem{Remark}{Remark}
\newtheorem*{theorem0}{Theorem 0 (El Soufi \cite{ElsoufiIndex})}
\newtheorem*{theorem1}{Theorem 1}
\newtheorem{Corollary}{Corollary}
\newtheorem*{theorem3}{Theorem 2}

\title{Harmonic maps from spheres with lowest possible index and their properties}
\author[]{Egor Surkov}
\date{}
\begin{document}
\begin{abstract}
 In this note, we prove that for $n > 2$ a harmonic map $\mathbb{S}^n \to \mathbb{CP}^m$,  of index  $n+1$ is pseudo horizontally weakly conformal, that $n$ must be odd. In particular, for even $n$ every non-constant harmonic map $\mathbb{S}^n \to \mathbb{CP}^m$ has index at least $n+2$. We also present a harmonic map from $\mathbb{S}^4$ to $\mathrm{Gr}_{\mathbb{C}}(2,4)$ of index $5$ that is not pseudo horizontally weakly conformal, so the statement does not extend to Grassmannians.
\end{abstract}

\maketitle
It is well known that the stability and instability of geodesics depend heavily on the curvature of the Riemannian manifold. They can be seen as the first geometrically interesting examples of harmonic maps. At the same time, if we consider geodesics as harmonic maps from the unit circle $\mathbb{S}^1$, their stability depends only on the curvature of the target manifold.

Harmonic maps from higher-dimensional spheres can be seen as a generalization of energy minimizing curves. For $n \geqslant 3$ the energy functional is no longer invariant under the large conformal group of $\mathbb{S}^n$, and this forces non-constant harmonic maps to be unstable. This was first established by Xin \cite{XinInstab}, who proved that for $n \geqslant 3$ every non-constant harmonic map from $\mathbb{S}^n$ is unstable, and by Smith \cite{SmithIndex}, who proved that for $n \geqslant 3$ the identity map of $\mathbb{S}^n$ has index $n+1$. Later, El Soufi \cite{ElsoufiIndex} proved that for $n \geqslant 3$ the index of every non-constant harmonic map from $\mathbb{S}^n$ is at least $n+1$, the index of the identity map. 

Hence, it is natural to ask which properties appear when harmonic maps have the minimal possible index. In the case when the domain manifold admits stable harmonic maps, it has been proven that stable maps have rigid properties \cite{Chen1}, \cite{Chen2}, \cite{StabGrass}; precisely, they are pseudo horizontally weakly conformal, see the definition below. For the embedded spheres, there is a recent result that says that spheres of minimal possible index are geodesic ones \cite{KetoverMinIndex}.

We prove that a harmonic map from $\mathbb{S}^n$ of minimal possible Morse index is pseudo horizontally weakly conformal when the target is $\mathbb{CP}^m$, and that it is a harmonic morphism when the target is $\mathbb{CP}^1 \cong \mathbb{S}^2$. We also present a counterexample showing that in the case of the complex Grassmannian $\Gr$ the situation is more complicated.

\subsection{Basic facts about harmonic maps}

Let us fix some notation and conventions. Consider a real vector bundle $E$ over a Riemannian manifold, equipped with a Euclidean scalar product and a metric connection $\nabla^{E}$. We define the trace Laplacian or Bochner Laplacian on $\Gamma(E)$ as
\begin{equation}
\label{trl}
\Delta_B^E = - \sum_\alpha \left( \nabla^{E}_{e_\alpha} \nabla^{E}_{e_\alpha} - \nabla^{E}_{\nabla_{e_\alpha}e_\alpha} \right),
\end{equation}
where $e_\alpha$ is a local orthonormal frame of the base and $\nabla$ is the Levi-Civita connection. Note that $\Delta_B^E$ is a nonnegative operator. We denote $R(\cdot\,,\cdot)^E$ defined by
\begin{equation}
\label{R}
R(X,Y)^{E}s = \nabla^{E}_X \nabla^{E}_Y s - \nabla^{E}_Y \nabla^{E}_X s - \nabla^{E}_{[X,Y]} s, \quad s \in \Gamma(E),
\end{equation}
the curvature of the connection $\nabla^{E}$. With this convention, the Ricci tensor is $\Ricc(Y) = \sum_i R(Y,e_i)e_i$ and $\langle R(X,Y)Y,X\rangle \geqslant 0$ on a manifold of nonnegative sectional curvature. Further, we often omit the superscript in $\nabla^{E}$ and $\Delta_B^E$, since $E$ is usually clear from the context. On the tangent bundle, we always consider the Levi-Civita connection with respect to the Riemannian metric. We deal mostly with Levi-Civita connections, pullbacks of Levi-Civita connections and the connections constructed from them by the natural bundle operations. We denote the Riemannian curvature tensor of a Riemannian manifold $(M,h)$ by $R^M$ instead of $R^{TM}$.

Now we recall several basic facts from the theory of harmonic maps, our main references are \cite{XinHarm} and \cite{HarmMorphism}.

Let $(M,h)$ be a closed Riemannian manifold and $\mathbb{S}^n$ the standard round $n$-dimensional sphere of radius $1$. Consider a map $\Psi \in C^{\infty}(\mathbb{S}^n,M)$. Recall that $d\Psi \in \Gamma(\mathrm{Hom}(T\mathbb{S}^n, \Psi^{*}TM)),$ and we can compute its Hilbert-Schmidt norm using the induced metric on the bundle $T^*\mathbb{S}^n \otimes \Psi^{*}TM$. We define the energy density $e(\Psi)$ of $\Psi$ as the squared length of $d\Psi$,
$$e(\Psi)=\|d\Psi\|^2 = g^{\alpha \beta} \frac{\partial \Psi^{i}}{\partial x^{\alpha}} \frac{\partial \Psi^{j}}{\partial x^{\beta}} h_{ij} = \sum_{\alpha=1}^n \langle d\Psi (e_\alpha), d\Psi(e_\alpha)\rangle_h,$$
where $e_\alpha$ is a local orthonormal frame in $\Gamma(T\mathbb{S}^n)$.

We associate with $\Psi$ the tension field $\tau(\Psi) \in \Gamma(\Psi^{*}TM)$ defined as
\begin{equation}
\label{tensionfield}
\tau(\Psi)=\tr(\nabla d\Psi) = \sum_\alpha \nabla d\Psi (e_\alpha, e_\alpha) = \sum_\alpha \Big[ \nabla_{e_\alpha}(d\Psi(e_\alpha)) - d\Psi(\nabla_{e_\alpha}e_\alpha) \Big].
\end{equation}

The energy functional is defined as
\begin{equation}
\label{energyFunc}
 E(\Psi) = \frac{1}{2}\int_{\mathbb{S}^n} e(\Psi)\, dV_{\mathbb{S}^n} = \frac{1}{2}\int_{\mathbb{S}^n} \|d\Psi\|^2\, dV_{\mathbb{S}^n}.
\end{equation}
The critical points of $E$ are called harmonic maps. Harmonic maps naturally appear as a generalization of classical harmonic functions. Under suitable curvature restrictions on the target, they can be seen as the nicest maps in their homotopy class, see \cite{EelsSimpson}.

It is well known that the Euler–Lagrange equation for the energy functional \eqref{energyFunc} is an elliptic equation
\begin{equation}
\label{Equation}
\tau(\Psi) = 0.
\end{equation}

Consider a two-parameter variation $\Psi_{s,t}: \mathbb{S}^n \times (-\varepsilon,\varepsilon) \times (-\varepsilon,\varepsilon) \to M$ with $\Psi_{0,0} = \Psi$, and let $v = \partial_s \Psi_{s,t}|_{(0,0)}$, $w = \partial_t \Psi_{s,t}|_{(0,0)} \in \Gamma(\Psi^*TM)$ be the variation fields. For a harmonic map $\Psi$ the second variation $\Hess E|_\Psi$ of the energy functional is
\begin{equation}
\label{SecondVar}
 \Hess E|_\Psi (v,w) := \frac{\partial^2 }{\partial t\, \partial s} E(\Psi_{s,t})\Big|_{(0,0)}
 = \int_{\mathbb{S}^n} \big\langle \trl v - \tr(R^{M}(v,d\Psi)d\Psi) ,\, w \big\rangle\, dV_{\mathbb{S}^n}.
\end{equation}
where
\begin{equation}
\label{trR}
\tr(R^M(v, d\Psi) d\Psi) = \sum_\alpha R^M(v, d\Psi (e_\alpha))d\Psi (e_\alpha).
\end{equation}

The operator in formula \eqref{SecondVar} is called the Jacobi operator,
\begin{equation}
\label{JacobiOp}
J(v) = \trl v - \tr(R^{M}(v,d\Psi)d\Psi),
\end{equation}
and it generalizes the Jacobi operator from the theory of geodesics.

We recall several facts about the second variation. Firstly, $\Hess E|_\Psi(\cdot\,,\cdot)$ is a symmetric bilinear form on $\Gamma(\Psi^{*}TM)$. Secondly, the Morse index of $\Psi$ is the dimension of a maximal subspace of $\Gamma(\Psi^{*}TM)$ on which the associated quadratic form $v \mapsto \Hess E|_\Psi(v,v)$ is negative definite, equivalently, it is the number of negative eigenvalues of $J$ counted with multiplicity. Thirdly, by the general theory of self-adjoint elliptic operators the index is finite, since $\mathbb{S}^n$ is compact.

\subsection{Special types of maps}

Consider a smooth map $\Psi: (N,g) \to (M,h)$ between Riemannian manifolds. At every point $x \in N$ one has the orthogonal decomposition $$T_xN = \mathcal{H}_x \oplus \mathcal{V}_x,$$ where $\mathcal{V}_x = \Ker d\Psi_x$ is called the vertical subspace and $\mathcal{H}_x = \mathcal{V}_x^{\perp}$ is called the horizontal subspace.
\begin{definition}
A map $\Psi$ is called weakly horizontally  conformal (WHC) if for every $x \in N$ either $d\Psi_x = 0$, or $d\Psi_x|_{\mathcal{H}_x}: \mathcal{H}_x \to T_{\Psi(x)}M$ is surjective and conformal, i.e. there is $\lambda(x) > 0$ such that $h(d\Psi(X),d\Psi(Y)) = \lambda(x)^2 g(X,Y)$ for all $X,Y \in \mathcal{H}_x$. The function $\lambda$ is called the dilation of $\Psi$.
\end{definition}
\begin{definition}
A map $\Psi: (N,g) \to (M,h)$ is called a harmonic morphism if it is harmonic and WHC, see \cite{HarmMorphism}).
\end{definition}

Suppose now that $(M,h,\J)$ is a Hermitian manifold with complex structure $\J$ and let $\varphi: (N,g) \to (M,h)$ be a smooth map. Denote by $d\varphi^*: T_{\varphi(x)}M \to T_xN$ the pointwise adjoint of $d\varphi_x$ with respect to $g$ and $h$, and extend $g$ and $d\varphi^*$ complex-bilinearly to the complexified tangent spaces.
\begin{definition}
The map $\varphi$ is called pseudo horizontally weakly conformal (PHWC) if
\begin{equation}
\label{PHWC}
g(d\varphi^*(Z),d\varphi^*(W)) = 0 \quad \text{for all } Z, W \in T^{1,0}M .
\end{equation}
\end{definition}
This is equivalent to each of the following conditions, see \cite{HarmMorphism}
\begin{enumerate}
\item[(1)] $g(d\varphi^*(\J X),d\varphi^*(\J Y)) = g(d\varphi^*(X),d\varphi^*(Y))$ for all real $X,Y \in TM$, equivalently, the endomorphism $d\varphi \circ d\varphi^*$ of $TM$ commutes with $\J$;
\item[(2)] for a basis $W_1, \dots, W_m$ of $T^{1,0}_{\varphi(x)}M$ one has $g(d\varphi^*(W_\alpha),d\varphi^*(W_\beta)) = 0$ for all $\alpha, \beta$;
\item[(3)] in local coordinates $x^i$ on $N$ and local holomorphic coordinates $z^\alpha$ on $M$,
\begin{equation}
\label{PHWC4}
 g^{ij}\frac{\partial \varphi^\alpha}{\partial x^i}\frac{\partial \varphi^\beta}{\partial x^j} = 0 \quad \text{for all } \alpha, \beta .
\end{equation}
\end{enumerate}
When $M$ is a Riemann surface, a symmetric endomorphism of $T_{\varphi(x)}M$ commuting with $\J$ is a multiple of the identity, so PHWC coincides with WHC. The detailed theory of harmonic morphisms can be found in the book \cite{HarmMorphism}.

\subsection{Results and motivations}
A natural question is what bounds on the index can be given in terms of the geometry of the target. If $M$ has nonpositive sectional curvature, then every harmonic map into $M$ is stable, see e.g. \cite{XinHarm}, combined with Theorem 0 below this shows that for $n \geqslant 3$ every harmonic map from $\mathbb{S}^n$ to such an $M$ is constant. In 1995 Ahmad El Soufi proved the following theorem.
\begin{theorem0}
Let $n \geqslant 3$ and let $\Psi$ be a harmonic map from $\mathbb{S}^n$ with the standard metric to a Riemannian manifold $(M,h)$. If $\Psi$ is not constant, then $\ind \Psi \geqslant n+1.$
\end{theorem0}
The restriction $n \geqslant 3$ is necessary, becasue, for example,  holomorphic maps $\mathbb{S}^2 \to \mathbb{S}^2$ are stable.

In the next section we present an alternative proof of the El Soufi theorem. The original proof uses a direct computation of the second derivative of $E$ along special variations; we give a proof based on the Bochner method. This approach was mentioned by El Soufi himself in \cite{ElsoufiIndex}, to the best of the author's knowledge it has not been published, and we would like to fill this gap.

In Section 2 we prove the following theorem, which is our main result. Recall that for $a \in \mathbb{R}^{n+1}$ we denote by $V_a$ the field given by \eqref{ConfField} on $\mathbb{S}^n$.
\begin{theorem1}
Let $n \geqslant 3$ and let $\Psi$ be a harmonic map from $\mathbb{S}^n$ to $\mathbb{CP}^m$ with the Fubini--Study metric such that $\ind \Psi = n+1$. Then
\begin{enumerate}
\item[(a)] $\Psi$ is PHWC;
\item[(b)] $n$ is odd;
\item[(c)] there exists $S \in \mathrm{End}(\mathbb{R}^{n+1})$ with $S^2 = -\Id$ such that
\begin{equation}
\label{Cintertwine}
\J\, d\Psi(V_a) = d\Psi(V_{Sa}) \qquad \text{for all } a \in \mathbb{R}^{n+1}.
\end{equation}
\end{enumerate}
\end{theorem1}
 As immediate corollaries, we obtain
\begin{Corollary}
Let $n \geqslant 3$ and let $\Psi$ be a harmonic map from $\mathbb{S}^n$ with the standard metric to the round sphere $\mathbb{S}^2$ with $\ind \Psi = n+1$. Then $n$ is odd and $\Psi$ is a harmonic morphism.
\end{Corollary}
\begin{Corollary}
Let $n \geqslant 4$ be even, and let $\Psi$ be a non-constant harmonic map from $\mathbb{S}^n$ to $\mathbb{CP}^m$, or to $\mathbb{S}^2$ equipped with any metric of positive Gaussian curvature. Then $\ind \Psi \geqslant n+2$.
\end{Corollary}
For $n=3$, Corollary 1 is the theorem of Rivi\`ere \cite{Riviere}. The round metric on $\mathbb{S}^2$ in Corollary 1 can be replaced by any metric of positive curvature, see Remark~\ref{remarkBumpMetr}. Corollary 2 improves El Soufi's bound for these targets. 

Our proof follows Rivi\`ere in the sense that the test fields for the shifted Hessian form are exactly the fields from El Soufi theorem after aplying complex structure. However, we use intrinsic calculations instead of the extrinsic method used by Rivi\`ere. 

In Section 3 we present the following example. 
\begin{theorem3}
There exists a harmonic map $\Psi$ from $\mathbb{S}^4$ to $\Grtf$ with $\ind \Psi = 5$ which is not PHWC.
\end{theorem3}
Essentially, Theorem 2 tells us that in the case of Grassmannians, the generalization of Theorems 1 and Corollary 1 is not straightforward, despite the fact that Chen's result on stable maps extends to Grassmannians \cite{StabGrass}.

\section{El Soufi theorem}
\subsection{Basic facts from differential geometry}
Let us recall basic facts about the Bochner-Weitzenböck technique for vector-bundle-valued differential forms, our basic reference is \cite[Chapter 1.3]{XinHarm}. Note that our sign convention \eqref{R} for the curvature is opposite to that of \cite{XinHarm}. All formulas below are stated in our convention.

Let $E$ be a real vector bundle with a metric and a connection compatible with it  over a Riemannian manifold $(N,g)$, and let $\omega$ be an $E$-valued differential $q$-form. Recall the operator $d: \Gamma(\Lambda^qT^*N \otimes E) \to \Gamma(\Lambda^{q+1}T^*N\otimes E)$,
\begin{equation}
\label{dop}
d\omega(S_0,\dots,S_q) = \sum_{k=0}^q (-1)^k(\nabla_{S_k}\omega)(S_0,\dots,\hat{S}_k,\dots,S_q),
\end{equation}
and its formal adjoint $\delta = d^*: \Gamma(\Lambda^qT^*N \otimes E) \to \Gamma(\Lambda^{q-1}T^*N \otimes E)$,
\begin{equation}
\label{deltaop}
\delta\omega(S_1,\dots,S_{q-1}) = - \sum_{k=1}^{\dim N}(\nabla_{e_k}\omega)(e_k,S_1,\dots,S_{q-1}),
\end{equation}
where $\nabla$ is the connection on $\Lambda^*T^*N \otimes E$ induced by the Levi-Civita connection and $\nabla^E$, and $S_i \in \Gamma(TN)$. The Hodge--Laplace operator on $\Gamma(\Lambda^{*}T^*N\otimes E)$ is
\begin{equation}
\label{HodgeL}
\Delta_H = d\delta + \delta d .
\end{equation}
Let $\Psi: (N,g) \to (M,h)$ be a smooth map and $E = \Psi^*TM$. Since the Levi-Civita connection is torsion-free, $d(d\Psi) = 0$, and by \eqref{deltaop} and \eqref{tensionfield} one has $\delta d\Psi = -\tau(\Psi)$. Hence $\Psi$ is harmonic if and only if $d\Psi$ is harmonic,
\begin{equation}
\label{HarmHod}
\Delta_H d\Psi = 0,
\end{equation}
note that for a closed $N$, $\Delta_H d\Psi = 0$ implies $\delta d\Psi = 0$ by integration by parts.

The Weitzenböck formula relates $\trl$ and $\Delta_H$ on $\Gamma(\Lambda^{*}T^*N\otimes E)$, see \cite[Chapter 1.3.1]{XinHarm}
\begin{equation}
\label{Wformula}
\Delta_H = \trl + S,
\end{equation}
where
\begin{equation}
\label{Wformulas}
(S\omega)(S_1,\dots,S_q) = \sum_{k,i} (-1)^{k+1}(R(e_i,S_k)\omega)(e_i,S_1,\dots,\hat{S}_k,\dots,S_q),
\end{equation}
and $R$ denotes the curvature of the induced connection on $\Lambda^{*}T^*N\otimes E$. Note that $R$ acts nontrivially on the $\Lambda^*T^*N$ factor of $\omega$. For a $1$-form this reads as $(S\omega)(Y) = \sum_i (R(e_i,Y)\omega)(e_i)$.

For $\omega = d\Psi$ and $Y \in \Gamma(TN)$ we have 
\[
(R(e_i,Y)d\Psi)(e_i) = R^{\Psi^*TM}(e_i,Y)(d\Psi(e_i)) - d\Psi(R^N(e_i,Y)e_i)
\]
with $R^{\Psi^*TM}(e_i,Y) = R^M(d\Psi(e_i),d\Psi(Y))$, and $\sum_i R^N(e_i,Y)e_i = -\Ricc^N(Y)$. Therefore
\begin{equation}
\label{Saction}
(Sd\Psi)(Y) = \sum_i R^M(d\Psi(e_i),d\Psi(Y))d\Psi(e_i) + d\Psi(\Ricc^N(Y)),
\end{equation}
see \cite[Formula (1.3.12)]{XinHarm}. 

\subsection{Case of the round sphere}
In this section, given a point $p$, we use a local orthonormal frame $e_i$ in $\Gamma(T\mathbb{S}^n)$ such that $\nabla_{e_i}e_i|_{p}= 0$
for local computations at the point $p$.

We are interested in the vector fields whose flows generate conformal transformations of the sphere, see e.g. \cite{ElSoufiConf}. For each fixed vector $a \in \mathbb{R}^{n+1}$ we associate a vector field $V_a \in \Gamma(T\mathbb{S}^n)$,
\begin{equation}
\label{ConfField}
V_a = a - \langle a,\nu \rangle \nu,
\end{equation}
where $\nu$ is the outward unit normal field of the canonical embedding $\mathbb{S}^n \subset \mathbb{R}^{n+1}$. The $V_a$ is the gradient of the restriction to $\mathbb{S}^n$ of the linear function $\langle a, \cdot\rangle$. It is known that the flows $\Phi_{V_a}^t$ together with the rotations of $\mathbb{R}^{n+1}$ generate the group of conformal diffeomorphisms of the sphere, so called Möbius group, which is isomorphic to $SO^{+}(n+1,1)$, see e.g. \cite{ElSoufiConf}, \cite{Helga}. Let $s_1,\dots,s_{n+1}$ be the canonical orthonormal basis of $\mathbb{R}^{n+1}$ and $x^1,\dots,x^{n+1}$ the corresponding Cartesian coordinates. We define the vector fields
\begin{equation}
\label{KeyField}
X_i = V_{s_i} = s_i - \langle s_i,\nu \rangle \nu = \grad (x^i|_{\mathbb{S}^n}),
\end{equation}
so that, in a local orthonormal frame,
\begin{equation}
\label{locfrX}
X_i = \sum_j\langle s_i, e_j\rangle e_j .
\end{equation}
Since $\partial_{e_k}\nu = e_k$ for the canonical embedding, we have
\begin{equation}
\label{nablaX}
\nabla_{e_k}X_i = \nabla_{e_k} ( s_i - \langle s_i,\nu \rangle \nu ) = -x^i e_k ,
\end{equation}
because all the other terms are normal to the sphere. Using \eqref{locfrX} and \eqref{nablaX} we get at the point $p$
\begin{equation}
\label{traceLaplX}
\trl X_i = - \sum_j \nabla_{e_j} \nabla_{e_j} X_i = \sum_j e_j(x^i)\, e_j = X_i,
\end{equation}
where we used $\nabla_{e_j}e_j = 0$ at $p$. Thus the $X_i$ are eigenfields of $\trl$ with eigenvalue $1$. Note also that
\begin{equation}
\label{XkId}
\sum_{k=1}^{n+1} \langle X_k, Y\rangle X_k = Y \qquad \text{for all } Y \in T\mathbb{S}^n,
\end{equation}
since $\sum_k \langle s_k, Y\rangle s_k = Y$ and $Y \perp \nu$.

\subsection{Bochner type proof of the El Soufi theorem}
\begin{proof}[Proof of El Soufi theorem]
Let $\Psi \in C^{\infty}(\mathbb{S}^n,M)$ be harmonic. The harmonicity condition \eqref{HarmHod}, the Weitzenböck formula \eqref{Wformula} and formula \eqref{Saction} give, for $Y \in \Gamma(T\mathbb{S}^n)$,
\begin{equation}
\label{Firstconsec}
\begin{aligned}
(\trl d\Psi)(Y) &= -(S d\Psi)(Y) = - \sum_i R^M(d\Psi(e_i),d\Psi(Y))d\Psi(e_i) - d\Psi(\Ricc^{\mathbb{S}^n} Y) \\
&= \tr R^M(d\Psi(Y),d\Psi)d\Psi - (n-1)\,d\Psi(Y),
\end{aligned}
\end{equation}
where $d\Psi \in \Gamma(T^*\mathbb{S}^n \otimes \Psi^*TM)$, and $\Psi^*TM$ carry the pullback metric and the pullback of the Levi-Civita connection. Our goal is to compute the action of the Jacobi operator on the sections
\[
\varkappa_k = d\Psi(X_k) \in \Gamma(\Psi^*TM), \qquad k = 1,\dots,n+1,
\]
where $X_k$ are the fields \eqref{KeyField}. Using the Leibniz rule for the induced connections, we get at the point $p$
\begin{equation}
\label{calculationW}
\begin{aligned}
\trl (\varkappa_k) & = - \sum_i \nabla_{e_i}\nabla_{e_i} \big( d\Psi(X_k) \big) \\
&= - \sum_i \Big[ (\nabla_{e_i}\nabla_{e_i} d\Psi)(X_k) + 2 (\nabla_{e_i} d\Psi)(\nabla_{e_i} X_k) + d\Psi(\nabla_{e_i}\nabla_{e_i} X_k) \Big] \\
& = (\trl d\Psi)(X_k) - 2 \sum_i (\nabla_{e_i} d\Psi)(\nabla_{e_i} X_k) + d\Psi(\trl X_k).
\end{aligned}
\end{equation}
By \eqref{nablaX}, \eqref{traceLaplX} and \eqref{Firstconsec} we obtain
\begin{equation}
\label{CalculationW2}
\begin{aligned}
\trl (\varkappa_k) &= \tr R^M(d\Psi(X_k),d\Psi)d\Psi - (n-1)\,d\Psi(X_k) + 2x^k \sum_i (\nabla_{e_i} d\Psi)(e_i) + d\Psi(X_k) \\
&= \tr R^M(d\Psi(X_k),d\Psi)d\Psi - (n-2)\,d\Psi(X_k) + 2x^k\, \tau(\Psi).
\end{aligned}
\end{equation}
Since $\Psi$ is harmonic, $\tau(\Psi) = 0$, and finally
\begin{equation}
\label{CalculationWF}
\trl (\varkappa_k) = \tr R^M(d\Psi(X_k),d\Psi)d\Psi - (n-2)\,d\Psi(X_k).
\end{equation}
Hence
\begin{equation}
\label{ResCalculation}
J(\varkappa_k) = \trl \varkappa_k - \tr(R^{M}(\varkappa_k,d\Psi)d\Psi) = -(n-2)\,\varkappa_k .
\end{equation}
So all $\varkappa_k$ are eigenfields of the Jacobi operator with eigenvalue $-(n-2)$, which is negative for $n \geqslant 3$. Since every linear combination of the $\varkappa_k$ is again an eigenfield with the same eigenvalue, $\Hess E|_\Psi$ is negative definite on $\mathrm{span}\{\varkappa_1,\dots,\varkappa_{n+1}\}$, and it remains to prove that the $\varkappa_k$ are linearly independent.

Suppose that $\sum_k c_k \varkappa_k = 0$ with $c = (c_1,\dots,c_{n+1}) \neq 0$. Put $a = \sum_k c_k s_k= c$. Since $a \mapsto V_a$ is linear, $d\Psi(V_a) = 0$. After a pre-rotation of $\mathbb{R}^{n+1}$ we may assume that $a = |a|\, s_1$, so that $d\Psi(X_1) = 0$, i.e. $\Psi$ is constant along the integral curves of $X_1 = \grad(x^1|_{\mathbb{S}^n})$. These integral curves are the open meridians joining the poles $-s_1$ and $s_1$, and every point of $\mathbb{S}^n \setminus \{\pm s_1\}$ lies on one of them. Since each meridian accumulates at $s_1$, continuity of $\Psi$ gives $\Psi(x) = \Psi(s_1)$ for all $x \neq -s_1$, and then also for $x = -s_1$. Thus, $\Psi$ is constant, a contradiction. This finishes the proof of the El Soufi theorem.
\end{proof}

\section{Maps to \texorpdfstring{$\mathbb{CP}^m$}{CP^m}}
Our main goal in this section is to prove Theorem 2, which extends the result of Chen \cite{Chen2} to the setting where the geometry of the domain forbids non-constant stable harmonic maps.

We use the Fubini-Study metric on $\mathbb{CP}^m$ normalized to have a constant holomorphic sectional curvature equal to $4$. It is convenient to recall the standard embedding of $\mathbb{CP}^m$ into the space $\Herm(m+1)$ of Hermitian $(m+1)\times(m+1)$ matrices, see e.g. \cite{canEmb}, \cite{RosStaff}, for $z = (z_0,\dots,z_m) \in \mathbb{C}^{m+1}$ with $\sum_i |z_i|^2 = 1$,
\[
[z_0:\dots:z_m] \mapsto A_z = (z_i \bar z_j)_{i,j} = z z^*,
\]
that is, $A_z$ is the orthogonal projection onto the line $\mathbb{C}z$; the image is
\[
\{A \in \Herm(m+1) \mid A^2 = A,\ \tr A = 1\}.
\]
With the scalar product $\langle A, B\rangle = \frac{1}{2}\tr(AB)$ on $\Herm(m+1)$ this embedding is isometric for the Fubini-Study metric of holomorphic sectional curvature $4$, and it is equivariant under the action $A \mapsto PAP^{-1}$ of $U(m+1)$. Particularly, $\mathbb{CP}^m \cong U(m+1)/(U(1) \times U(m))$ is a Hermitian symmetric space. The tangent and normal spaces at $A$ are
\[
T_A\mathbb{CP}^m = \{ X \in \Herm(m+1) \mid XA + AX = X \}\]
\[N_A\mathbb{CP}^m = \{ Y \in \Herm(m+1) \mid AY = YA \}.
\]
We shall not need the extrinsic geometry of this embedding, the proof of Theorem 2 is intrinsic and uses only the parallelism of the complex structure $\J$ and the curvature tensor of $\mathbb{CP}^m$, which in our convention is
\begin{equation}
\label{CPcurv}
R(X,Y)Z = \langle Y,Z\rangle X - \langle X,Z\rangle Y + \langle \J Y,Z\rangle \J X - \langle \J X,Z\rangle \J Y + 2\langle X,\J Y\rangle \J Z,
\end{equation}
so that
\begin{equation}
\label{CPsec}
\langle R(V,U)U,V \rangle = |V|^2|U|^2 - \langle V,U\rangle^2 + 3\langle \J V,U\rangle^2 .
\end{equation}

\begin{proof}[Proof of Theorem 1]
Let $n \geqslant 3$. 
Recall that $\varkappa_k = d\Psi(X_k)$, $k = 1,\dots,n+1$, are linearly independent eigenfields of Jacobi operator $J$ with eigenvalue $-(n-2)$, see Section 2.3. Since $\ind \Psi = n+1$, they span the negative eigenspace of $J$, and all other eigenvalues of $J$ are nonnegative, hence $J + (n-2) \geqslant 0$. Define
\begin{equation}
\label{Qdef}
Q(w) := \Hess E|_\Psi(w,w) + (n-2)\int_{\mathbb{S}^n} |w|^2\, dV, \qquad w \in \Gamma(\Psi^*T\mathbb{CP}^m).
\end{equation}
Then $Q \geqslant 0$, and $Q(\varkappa_k) = 0$ for every $k$.

Fix a point $x \in \mathbb{S}^n$, an orthonormal basis $e_1,\dots,e_n$ of $T_x\mathbb{S}^n$, and set
\[
u_i = d\Psi(e_i), \qquad A_\Psi = d\Psi \circ d\Psi^* \in \mathrm{End}(T_{\Psi(x)}\mathbb{CP}^m),
\]
where $d\Psi^*$ is the pointwise adjoint of $d\Psi$. By \eqref{XkId}, $\sum_k X_k \otimes X_k^* = \Id_{T\mathbb{S}^n}$, where $Y^* = \langle Y, \cdot\rangle$, hence
\begin{equation}
\label{Akappa}
A_\Psi = \sum_{i=1}^n u_i \otimes u_i^* = \sum_{k=1}^{n+1} \varkappa_k \otimes \varkappa_k^* .
\end{equation}
Since $\mathbb{CP}^m$ is Kähler, $\nabla \J = 0$, and therefore
\[
|\nabla(\J\varkappa_k)| = |\nabla \varkappa_k|, \qquad |\J\varkappa_k| = |\varkappa_k| ,
\]
so the difference $Q(\J\varkappa_k) - Q(\varkappa_k)$ comes only from the curvature term in \eqref{SecondVar}. From \eqref{CPsec} we get
\begin{equation}
\label{CPdiff}
\langle R(\J V,U)U,\J V\rangle - \langle R(V,U)U,V\rangle = 4\big( \langle V,U\rangle^2 - \langle \J V,U\rangle^2 \big),
\end{equation}
and therefore
\begin{equation}
\label{Qdiff}
Q(\J\varkappa_k) - Q(\varkappa_k) = -4\int_{\mathbb{S}^n} \sum_i \Big( \langle \varkappa_k,u_i\rangle^2 - \langle \J\varkappa_k,u_i\rangle^2 \Big) dV .
\end{equation}
By \eqref{Akappa},
\[\sum_{i,k} \langle \varkappa_k,u_i\rangle^2 = \sum_i \langle A_\Psi u_i, u_i\rangle = \tr(A_\Psi^2), \]
\[
\sum_{i,k} \langle \J\varkappa_k,u_i\rangle^2 = \sum_i \langle A_\Psi \J u_i, \J u_i\rangle = -\tr(A_\Psi \J A_\Psi \J).
\]
Since $A_\Psi$ is symmetric and $\J$ is skew-symmetric, $[A_\Psi,\J]$ is symmetric and \[|[A_\Psi,\J]|^2 = \tr([A_\Psi,\J]^2) = 2\tr(A_\Psi^2) + 2\tr(A_\Psi \J A_\Psi \J).\] Summing \eqref{Qdiff} over $k$ and using $Q(\varkappa_k) = 0$ we obtain
\begin{equation}
\label{Qsum}
\sum_{k=1}^{n+1} Q(\J\varkappa_k) = -4\int_{\mathbb{S}^n} \Big( \tr(A_\Psi^2) + \tr(A_\Psi \J A_\Psi \J) \Big) dV = -2\int_{\mathbb{S}^n} |[A_\Psi,\J]|^2\, dV .
\end{equation}
The left-hand side of \eqref{Qsum} is nonnegative since $Q \geqslant 0$, and the right-hand side is nonpositive. Hence both sides vanish. In particular $[A_\Psi, \J] = 0$ on $\mathbb{S}^n$, i.e.
$$d\Psi \circ d\Psi^* \circ \J = \J \circ d\Psi \circ d\Psi^* ,$$
which is (a) by condition (1) in the definition of PHWC.

Let us note that we can extract more information from \eqref{Qsum}. Every term of  its left-hand side is nonnegative and the sum vanishes, hence
\begin{equation}
\label{QJkappa}
Q(\J\varkappa_k) = 0 \qquad \text{for all } k = 1,\dots,n+1 .
\end{equation}
Define the following operator $L = J + (n-2)$. This is a self-adjoint elliptic operator on $\Gamma(\Psi^*T\mathbb{CP}^m)$ with $L \geqslant 0$, and $Q(w) = \int_{\mathbb{S}^n} \langle Lw, w\rangle\, dV$. By \eqref{QJkappa},
\[
\J\varkappa_k \in \Ker L \qquad \text{for all } k .
\]
Kernel of $L$ contains
$
\mathcal{K} := \mathrm{span}_{\mathbb{R}}\{\varkappa_1,\dots,\varkappa_{n+1}\},
$
which has dimension $n+1$. Thus we get $\Ker L = \mathcal{K}$, and we already saw that
\begin{equation}
\label{JK}
\J \mathcal{K} \subset \mathcal{K} .
\end{equation}

Now we can define the following map
\[
\Lambda : \mathbb{R}^{n+1} \to \mathcal{K}, \qquad \Lambda(a) = d\Psi(V_a) = \sum_{k} a_k \varkappa_k ,
\]
which is an isomorphism, by linear independence of the $\varkappa_k$. By \eqref{JK} we can define the following endomorphism
\[
S := \Lambda^{-1} \circ \J \circ \Lambda \in \mathrm{End}(\mathbb{R}^{n+1})
\]
 and by construction $\J\, d\Psi(V_a) = \J\Lambda(a) = \Lambda(Sa) = d\Psi(V_{Sa})$, which is \eqref{Cintertwine}. Moreover, $S^2 = \Lambda^{-1} \J^2 \Lambda = -\Id$, this proves (c). 

For (b), note that $S$ is a real $(n+1)\times(n+1)$ matrix, so $\det S \in \mathbb{R}$ and
$$(\det S)^2 = \det(S^2) = \det(-\Id_{\mathbb{R}^{n+1}}) = (-1)^{n+1},$$
which forces $n+1$ to be even.

Finally, we make a remark that since $A_\Psi = d\Psi \circ d\Psi^*$ commutes with $\J$, its image is $\J$-invariant. In addition, we have $\langle A_\Psi v, v\rangle = |d\Psi^* v|^2$, and also $\Ker A_\Psi = \Ker d\Psi^*$. Hence $\text{Im} A_\Psi = (\Ker d\Psi^*)^{\perp} = \text{Im} d\Psi_x = d\Psi_x(T_x\mathbb{S}^n)$. Thus, $d\Psi_x(T_x\mathbb{S}^n)$ is a complex subspace of $T_{\Psi(x)}\mathbb{CP}^m$, and its real dimension is even, hence the dimension is at most $n-1$ because $n$ is odd. This finishes the proof of Theorem 1.
\end{proof}
Corollary 2 immediately follows.
\begin{Remark}\label{remarkBumpMetr}
The proof uses only that $\J$ is parallel and the identity \eqref{CPdiff}. If the target is an oriented surface $\Sigma$ with Gaussian curvature $K$ and $\J$ is the rotation by $\pi/2$, then $\nabla \J = 0$ and $\langle R(\J V,U)U,\J V\rangle - \langle R(V,U)U,V\rangle = K\big( \langle V,U\rangle^2 - \langle \J V,U\rangle^2 \big)$, and the same computation gives
$$\sum_{k=1}^{n+1} Q(\J\varkappa_k) = -\frac{1}{2}\int_{\mathbb{S}^n} K\, |[A_\Psi,\J]|^2\, dV .$$
Hence, if $K > 0$ is everywhere, a harmonic map $\Psi: \mathbb{S}^n \to \Sigma$ of index $n+1$ satisfies $[A_\Psi,\J] = 0$, which for a target being a surface means that $\Psi$ is WHC. In particular, Corollary 1 is the case $m = 1$ of Theorem 2. Because the round metric of $\mathbb{S}^2$ is a constant multiple of the Fubini-Study metric of $\mathbb{CP}^1$, the index and the harmonic morphism property are not affected by a constant rescaling of the target metric. Thus, we have that Corollary 1 remains true for every metric of positive curvature on $\mathbb{S}^2$. Compare this with the result of Riviere \cite{Riviere}.
\end{Remark}

\section{One map to $\Grtf$}
In this section, we show that the extension of the result of \cite{StabGrass} to the case of the domain $\mathbb{S}^n$ is not straightforward.

\subsection{Geometry of $\Gr$}
We recall basic facts about the complex Grassmannian $\Gr$ of $k$-dimensional complex subspaces of $\mathbb{C}^{N}$ in the matrix model, in order to fix conventions and for the convenience of the reader, all facts below are classical, see e.g. \cite{Helga}. We have
\[
\Gr \;\cong\; \{\, P\in \mathrm{Mat}_{N\times N}(\mathbb{C}) \mid P^2=P,\; P^\ast = P,\; \tr(P)=k \,\},
\]
the space of Hermitian projections of rank $k$. The tangent space at a point $P$ is
\[
T_P \Gr = \{\, X\in \mathrm{Mat}_{N\times N}(\mathbb{C}) \mid X^\ast = X,\; XP + PX = X \,\} = \{\, [A,P] \mid A \in \mathfrak{u}(N) \,\}.
\]
The Grassmannian inherits a $\mathrm{U}(N)$-invariant metric from the embedding into the space of Hermitian matrices, for $X,Y \in T_P\Gr$
\[
\langle X,Y \rangle_P = \tr(X Y^*) = \tr(XY),
\]
the last equality holds because $Y^* = Y$. This is the restriction of the Hilbert-Schmidt metric on Hermitian matrices, and $\Gr \cong \mathrm{U}(N)/(\mathrm{U}(k)\times\mathrm{U}(N-k))$ is a compact symmetric space with this metric. Note that for $k = 1$ this metric is twice the metric used in Section 2. But the index and the PHWC property are not affected  by a constant rescaling of the target metric.

The Levi-Civita connection is obtained by projecting the ambient derivative to the tangent space,
\[
\nabla_X Y = \pi_T\big( dY(X) \big), \qquad \pi_T(Z) = [P,[P,Z]] = PZ + ZP - 2PZP ,
\]
where $\pi_T$ is the orthogonal projection of a Hermitian matrix $Z$ onto $T_P\Gr$.

The Grassmannian is a Hermitian symmetric space. Its complex structure in the matrix model is
\[
\J_P(X) = i\,[P,X] = i\,(2P - I)X, \qquad X \in T_P\Gr .
\]
Indeed, $i[P,X]$ is Hermitian and tangent at $P$, and $\J_P^2 X = -[P,[P,X]] = -\pi_T(X) = -X$. The metric is Hermitian, $\langle \J X, \J Y\rangle = \langle X,Y\rangle$, the Kähler form is $\omega(X,Y) = \langle \J X, Y\rangle$, and since $\Gr$ is a Kähler homogeneous manifold, $\J$ is parallel, $\nabla \J = 0$.

Being a compact Kähler symmetric space, $\Gr$ has curvature given by a commutator formula, see e.g. \cite{Helga}, for $X,Y,Z\in T_P \Gr$,
\begin{equation}
\label{GrCurv}
R(X,Y)Z = [[X,Y],Z] .
\end{equation}
In particular $\langle R(X,Y)Y,X\rangle = -\tr([X,Y]^2) = |[X,Y]|^2 \geqslant 0$, the sectional curvature is nonnegative, and it is strictly positive on holomorphic two-planes.

\subsection{Non-extension of the results to the Grassmannian case}

To present the map $\Psi$ of Theorem 2 we use the quadric model of $\Grtf$. The Plücker embedding identifies $\Grtf$ with the Klein quadric
\[
Q^4 = \{ [z] \in \mathbb{CP}^5 \mid z \cdot z = 0 \}, \qquad z \cdot w = \sum_{i=0}^5 z_i w_i ,
\]
where $z \cdot w$ is the complex bilinear form on $\mathbb{C}^6$. Throughout this section, we denote by $g_0$ the metric induced on $Q^4 \cong \Grtf$ by the Fubini–Study metric of $\mathbb{CP}^5$ of holomorphic sectional curvature $4$. The metric $g_0$ is $\mathrm{U}(4)$-invariant, hence, it is a symmetric metric on $\Grtf$, and it is a constant multiple of the matrix-model metric of previous subsection, in fact, it equals $2g_0$. Since the index and the PHWC property are invariant under constant rescalings of the target metric, we work with $g_0$.

Write $\mathbb{R}^6 = \mathbb{R}e_0 \oplus \mathbb{R}^5$ and let $\mathbb{S}^4 \subset \mathbb{R}^5$ be the unit sphere. Define
\[
\Psi: \mathbb{S}^4 \to \Grtf \cong Q^4 \subset \mathbb{CP}^5, \qquad \Psi(x) = [e_0 + ix].
\]
Since $(e_0+i x)\cdot(e_0+i x) = |e_0|^2-|x|^2 = 0$, we have $\Psi(x)\in Q^4$.

This map is very well-known in the theory of totally geodesic embeddings into symmetric spaces \cite{ChenNagano1}.

\subsubsection{The induced metric}

Lift $\Psi$ to the Hopf fibration, i.e. unit sphere $\mathbb{S}^{11} \subset \mathbb{C}^6$
\[
\widehat\Psi(x) = \frac{1}{\sqrt2}(e_0+i x).
\]
For $X\in T_x\mathbb S^4$ we have $d\widehat\Psi(X) = \frac{i}{\sqrt2}X$ and
\[
\big\langle \widehat\Psi(x),d\widehat\Psi(X) \big\rangle_{\mathbb C} = \frac{1}{2}\big( -i\langle e_0, X\rangle + \langle x, X\rangle \big) = 0,
\]
so $d\widehat\Psi(X)$ is horizontal with respect to the Hopf fibration $\mathbb{S}^{11} \to \mathbb{CP}^5$, which is a Riemannian submersion onto $(\mathbb{CP}^5, g_{FS})$. Therefore
\begin{equation}
\label{PsiHomothetic}
\Psi^*g_0 (X,X) = |d\widehat\Psi(X)|^2 = \tfrac{1}{2}\, g_{\mathbb S^4}(X,X) :
\end{equation}
$\Psi$ is a homothetic immersion with dilation $1/\sqrt2$, or equivalently, an isometric immersion for the matrix-model metric $2g_0$.

Now define $\rho \in O(6)$ by
\[
\rho(e_0)=e_0, \qquad \rho|_{e_0^\perp}=-\Id,
\]
and
\[
\sigma: Q^4 \longrightarrow Q^4, \qquad \sigma([z])=[\rho\bar z].
\]
Since $\rho$ is real orthogonal, $\sigma$ preserves the quadric and is an anti-holomorphic isometry of $(Q^4,g_0)$. Moreover, for $x\in\mathbb S^4$,
\[
\rho\overline{(e_0+i x)} = \rho(e_0-i x) = e_0+i x,
\]
hence $\Psi(\mathbb S^4) \subset \Fix(\sigma)$. A connected component of the fixed point set of an isometry is a totally geodesic submanifold. Since $\sigma$ is anti-holomorphic, its fixed point set is totally real, hence, has a real dimension of at most $\dim_{\mathbb{C}} Q^4 = 4$. Thus, $\Psi(\mathbb{S}^4)$ is a totally geodesic Lagrangian submanifold. Being a homothetic immersion into a totally geodesic submanifold, $\Psi$ satisfies $\nabla d\Psi = 0$, in particular, $\Psi$ is harmonic.

\subsubsection{The map $\Psi$ is not PHWC}

Let
\[
H = d\Psi(T\mathbb S^4) \subset \Psi^*T\Grtf .
\]
Since the image is Lagrangian, $\J H = H^\perp$. Let $A_\Psi = d\Psi\circ d\Psi^*$ be as in Section 2, the adjoint being taken with respect to $g_{\mathbb{S}^4}$ and $g_0$. By \eqref{PsiHomothetic},
\[
A_\Psi = \tfrac12\, P_H ,
\]
where $P_H$ denotes the orthogonal projection onto $H$. For $0\neq V\in H$ we have $\J V\in H^\perp$, so $A_\Psi(\J V)=0$, while $\J A_\Psi(V) = \frac12 \J V$. Consequently
\[
[A_\Psi,\J]V = A_\Psi(\J V)-\J A_\Psi(V) = -\tfrac12\,\J V\neq0 ,
\]
and $\Psi$ is not PHWC by condition (1) in the definition.

\subsubsection{Computation of the Morse index}

Recall that the curvature tensor of the complex quadric $(Q^4,g_0)$ in our convention is
\begin{align}
R(X,Y)Z ={}& \langle Y,Z\rangle X - \langle X,Z\rangle Y
+ \langle \J Y,Z\rangle \J X - \langle \J X,Z\rangle \J Y - 2\langle \J X,Y\rangle \J Z \notag\\
&+ \langle \mathcal{A}Y,Z\rangle \mathcal{A}X - \langle \mathcal{A}X,Z\rangle \mathcal{A}Y
+ \langle \J\mathcal{A}Y,Z\rangle \J\mathcal{A}X - \langle \J\mathcal{A}X,Z\rangle \J\mathcal{A}Y,
\label{QuadricCurvature}
\end{align}
where $\mathcal{A}$ is a local real structure on $TQ^4$, determined up to a $S^1$ factor and satisfies
\[
\mathcal{A}^2=\Id, \qquad \mathcal{A}\J=-\J\mathcal{A}.
\]
At the point $[e_0 + ix]$ the tangent space of $Q^4$ is identified with $T_x\mathbb{S}^4 \otimes \mathbb{C} \subset \mathbb{C}^5$, i.e. the vectors horizontal for the Hopf fibration and tangent to the cone $z \cdot z = 0$, the real structures are $v \mapsto \lambda \bar v$, $\lambda \in S^1$, and $d\sigma(v) = \rho \bar v = -\bar v$. Along the  $\Psi(\mathbb S^4) = \Fix(\sigma)$ we choose $\mathcal{A} = d\sigma$. We have that the tangent bundle $H$ of the image of the embedding is the $(+1)$-eigenspace of $\mathcal{A}$,
\[
\mathcal{A}|_H=\Id, \qquad \mathcal{A}|_{\J H}=-\Id.
\]

Let $e_1,\ldots,e_4$ be a local orthonormal frame on the unit sphere and $u_i = d\Psi(e_i)$, so that $|u_i|^2_{g_0} = \frac12$ by \eqref{PsiHomothetic}. Put $E_i=\sqrt2\,u_i$. Thus, we have $E_1,E_2,E_3,E_4$, which is an orthonormal frame of $H$.

Let $W\in H$. Using \eqref{QuadricCurvature} together with $\mathcal{A}E_i=E_i$, $\mathcal{A}W=W$ and $\langle \J W, E_i\rangle = \langle \J E_i, E_i \rangle = 0$, we obtain
\[
R(W,E_i)E_i = 2\left( W-\langle W,E_i\rangle E_i \right).
\]
Therefore
\begin{equation}
\label{CurvTangentialRealForm}
\sum_{i=1}^4R(W,E_i)E_i = 6W \qquad \text{and} \qquad \sum_{i=1}^4R(W,u_i)u_i = 3W .
\end{equation}

 For $\J W\in \J H$ we have $\mathcal{A}(\J W)=-\J W$, and formula \eqref{QuadricCurvature} gives
\[
R(\J W,E_i)E_i = 2\langle W,E_i\rangle \J E_i .
\]
Summing over $i$, we obtain
\begin{equation}
\label{CurvNormalRealForm}
\sum_{i=1}^4 R(\J W,E_i)E_i = 2\J W \qquad \text{and} \qquad \sum_{i=1}^4 R(\J W,u_i)u_i = \J W .
\end{equation}

Since the image of embedding is totally geodesic and $\J$ is parallel, the orthogonal decomposition
\[
\Psi^*T\Grtf = H\oplus \J H
\]
is parallel and could be done in the level of bundles. It follows that every section can be written uniquely as
\[
V=d\Psi(X)+\J d\Psi(Y), \qquad X,Y\in\Gamma(T\mathbb S^4).
\]
Since $\nabla d\Psi=0$ and $\nabla \J = 0$, we have
\[
\Delta_B^\Psi(d\Psi(X)) = d\Psi(\Delta_B^{\mathbb S^4}X), \qquad \Delta_B^\Psi(\J d\Psi(Y)) = \J d\Psi(\Delta_B^{\mathbb S^4}Y).
\]
We recall that
\[
J_\Psi(V) = \Delta_B^\Psi V - \sum_{i=1}^4R(V,u_i)u_i .
\]
Using \eqref{CurvTangentialRealForm} and \eqref{CurvNormalRealForm} we obtain
\begin{equation}
\label{JacobiTangentialRealForm}
J_\Psi(d\Psi(X)) = d\Psi\left( \Delta_BX-3X \right), \qquad J_\Psi(\J d\Psi(Y)) = \J d\Psi\left( \Delta_BY-Y \right).
\end{equation}
Thus, under the identification $\Psi^*T\Grtf \simeq T\mathbb S^4\oplus T\mathbb S^4$, the Jacobi operator is
\begin{equation}
\label{JacobiSplitRealForm}
J_\Psi = (\Delta_B-3) \oplus (\Delta_B-1).
\end{equation}
Note that $\Delta_B X_k = X_k$ by \eqref{traceLaplX}, so $J_\Psi(d\Psi(X_k)) = -2\, d\Psi(X_k) = -(n-2)\, d\Psi(X_k)$ with $n = 4$, in agreement with \eqref{ResCalculation}.

It remains to recall the spectrum of $\Delta_B$ on vector fields of $\mathbb S^4$. We can identify vector fields and one-forms  using the metric. By the Weitzenböck formula, on one-forms of $\mathbb{S}^4$
\[
\Delta_H = \Delta_B + \Ricc = \Delta_B+3 .
\]
We can use a classical calculation for the laplacian on the forms on $\mathbb{S}^n$ given in \cite{Taneguchi}. In our case, we have that $\Delta_B \geqslant 1$ on one-forms of $\mathbb{S}^4$, and the eigenvalue $1$ being attained exactly on our fields $X_k$ for $k = 1,2,3,4,5$.  Hence, $\Delta_B - 1 \geqslant 0$, while $\Delta_B - 3$ is negative exactly on the $5$-dimensional space spanned by the $X_k$ . Therefore
$\ind(\Psi)=5$

One could check that the nullity of $\Psi$ equals $10 + 5 = 15$, which is consistent with the $15$-dimensional isometry group of $\Grtf$.

This finishes the proof of Theorem 3.

\begin{center}
\textbf{Acknowledgment}
\end{center}
The author is grateful to Alexei Penskoi for drawing his attention to this problem, for fruitful discussions, constant support, and help in the preparation of this paper. The author is also grateful to Mikhail Karpukhin and Iosif Polterovich for their valuable remarks on the manuscript and for pointing out the error in an earlier version of the text. The work was partially supported by the scholarship of the Theoretical Physics and Mathematics Advancement Foundation ``BASIS'' (No.~24-8-2-10-1).

\textbf{Usage of AI:} The map and reference that were used to construct an example in Section 3 were found with the help of ChatGPT 5,6; all calculations are by the author, who takes full responsibility for the mathematical content of the paper. \textbf{GenAI}  was used to improve the  structure of the text and for proofreading at the final stage. 

\bibliographystyle{alphaurl}
\bibliography{HarmonicMaps}
\end{document}